\documentclass[12pt,intlimits]{article}
\usepackage[a4paper,top=2cm,bottom=2cm,left=2cm,right=2cm,bindingoffset=5mm]{geometry}
\usepackage{amsmath}
\usepackage{amsthm}
\usepackage{amssymb}

\usepackage{graphicx}
\usepackage{epsfig}
\usepackage{hyperref}
\hypersetup{colorlinks=true,linkcolor=blue}
\newtheorem{theo}{Theorem}[section]
\newtheorem{lem}[theo]{Lemma}
\newtheorem{defi}[theo]{Definition}

\newtheorem{prop}[theo]{Proposition}

\newcommand{\be}{\begin{eqnarray}}
\newcommand{\ee}{\end{eqnarray}}
\newcommand{\bes}{\begin{eqnarray*}}
\newcommand{\ees}{\end{eqnarray*}}
\newcommand{\bi}{\begin{itemize}}
\newcommand{\ei}{\end{itemize}}
\newcommand{\ben}{\begin{enumerate}}
\newcommand{\een}{\end{enumerate}}

\newcommand{\La}{\mathcal{L}}

\newcommand{\E}{\mathcal{E}}

\newcommand{\R}{\mathbb{R}}
\newcommand{\N}{\mathbb{N}}

\newcommand{\G}{\mathcal{G}}

\newcommand{\de}{\mathrm {d}}
\newcommand{\Lip}{\mathrm {Lip}}

\def\Ext{{\hbox{\rm Ext}}}

\def\einschr{\hbox{\kern1pt\vrule height 6pt\vrule  width6pt height 0.4pt depth0pt\kern1pt}}

\newcommand {\no} {\noindent}

\newcommand{\enne}{{(n)}}

\DeclareMathOperator{\dive}{div}

\DeclareMathOperator{\Intern}{Int}

\title{\bf  Singular $p$-homogenization for highly conductive fractal layers  }

\date{}

\begin{document}
\maketitle

\centerline{\scshape Simone Creo}
\medskip
{\footnotesize

 \centerline{Dipartimento di Scienze di Base e Applicate per l'Ingegneria, Sapienza Universit\`{a} di Roma,
}
   \centerline{Via A. Scarpa 16,}
   \centerline{00161 Roma, Italy.}
}

\vspace{1cm}

\begin{center}

\small

\noindent First published in: Creo Simone, \lq\lq Singular $p$-homogenization for highly conductive fractal layers". Z. Anal. Anwend., in press (2021). \copyright\,European Mathematical Society.

\end{center}

\vspace{0.5cm}

\begin{abstract}
\noindent We consider a quasi-linear homogenization problem in a two-dimensional pre-fractal domain $\Omega_n$, for $n\in\N$, surrounded by thick fibers of amplitude $\varepsilon$. We introduce a sequence of $\lq\lq$pre-homogenized" energy functionals and we prove that this sequence converges in a suitable sense to a quasi-linear fractal energy functional involving a $p$-energy on the fractal boundary. We prove existence and uniqueness results for (quasi-linear) pre-homogenized and homogenized fractal problems. The convergence of the solutions is also investigated.
\end{abstract}

\medskip



\bigskip

\section*{Introduction}
\setcounter{equation}{0}

In the last decades, the study of different boundary value problems in domains with irregular boundary (or presenting irregular interfaces) has been of great interest. This is due to the fact that many industrial processes and natural phenomena occur across highly irregular media. It turns out that fractals can well model such irregular geometries; we remark that fractal sets are constructed by an iterative process.\\
Different problems on various fractal domains have been studied in the literature: among the others, we refer to \cite{lan-zei,moscoviv,lanciaviv,lanviv2,LRDV,frazionario}. In particular, in the latest years the focus has been on the so-called \emph{Venttsel' problems}, known in literature also as problems with dynamical boundary conditions.\\
Mathematically, Venttsel' problems are characterized by an unusual boundary condition: the operators governing the diffusion in the bulk and on the boundary are of the same order. For the literature on Venttsel' problems in piecewise smooth or fractal domains, from linear to quasi-linear, from local to nonlocal, we refer to \cite{JEE,CPAA,miovalerio,LVSV,nostromassimo,ambprodi,nostronazarov,CLNFCAA}.\\
In the linear smooth case, it is by now well known that Venttsel' problems can be seen as the limit of suitable homogenization problems. In particular, as in the seminal paper \cite{sanpal} (see also \cite{Att,FLLM}), the authors consider a boundary value problem in a domain containing a thin strip of large conductivity. If the product between the thickness and the conductivity of the strip has finite non-zero limit, it is proved that the limit problem is of Venttsel' type.\\
This result, in the linear case, has been extended to the case of fractal-type domains; among the others, we refer to \cite{lanmos,vanishing,mosvivthin,moscoviv2015,capvivsiam}. The authors consider transmission problems in regular domains of $\R^N$, containing a pre-fractal interface coated with a thin fiber of amplitude $\varepsilon$ and they prove, under suitable structural assumptions on the fiber, that the limit problem presents a $\lq\lq$Venttsel'-type" condition on the limit fractal interface.\\
The aim of the present paper is to prove that quasi-linear $\lq\lq$pure" Venttsel' problems in two-dimensional fractal domains can be seen as limit of suitable homogenization problems, thus generalizing the results known for the smooth quasi-linear case considered in \cite{aitmoussa} and providing a better physical ground to Venttsel' problems. To our knowledge, the present paper is the first example of homogenization of quasi-linear problems in the fractal case. We point out that quasi-linear homogenization problems model nonlinear thermal conduction in highly conductive thin structures.


The key issue is to suitably construct the fiber of thickness $\varepsilon$ in order to obtain in the limit the Venttsel' boundary condition. In the smooth case, the usual homogenization technique is to approximate a one-dimensional infinitely conductive thin layer by a two-dimensional thin layer of vanishing thickness $\varepsilon$ and increasingly high conductivity $a_\varepsilon$.
However, the construction of an $\varepsilon$-neighborhood for a fractal layer is tricky.\\
Following \cite{moscovivgeo} (see also \cite{lanmos}), we construct a thin fiber around pre-fractal domains. Pre-fractal domains have piecewise smooth boundary of polygonal type, and they depend on a natural parameter $n\in\N$, denoting the order of the iteration process in the construction of the fractal.

More precisely, we formally state the pre-homogenized problem as follows:
\begin{equation*}
(P_\varepsilon^n)
\begin{cases}
-\dive(a_\varepsilon^n(x,y)|\nabla u|^{p-2}\nabla u)+|u|^{p-2}u=f &\text{in}\,\,\Omega_\varepsilon^n,\\[2mm]
[u]=0 &\text{on}\,\,\partial\Omega_n\text{ and on}\,\,\Gamma_\varepsilon^n,\\[2mm]
\textrm{suitable transmission conditions}\quad &
\end{cases}
\end{equation*}%
where $\Omega_\varepsilon^n$ is a suitable pre-fractal domain $\Omega_n$ surrounded by a fiber of thickness $\varepsilon$ sufficiently small (see Section \ref{preliminari} for the details), $f$ is a given function in a suitable Lebesgue space, $\Gamma_\varepsilon^n$ will be suitably defined in Section \ref{preliminari} and $a^n_\varepsilon(x,y)$ is the conductivity of $\Omega^n_{\varepsilon}$ which will be defined later. Problem $(P_\varepsilon^n)$ presents also sophisticated transmission conditions which will be satisfied in a suitable weak sense, see Section \ref{exunconv}.\\
We introduce a sequence of quasi-linear energy functionals defined on $\Omega_\varepsilon^n$. Our aim is to prove that the pre-homogenized functionals \emph{M-converge} to a limit fractal energy functional. This convergence is very delicate since it is driven by two parameters $n$ and $\varepsilon$ and we are interested in the limit as $\varepsilon\to 0$ and $n\to+\infty$ simultaneously. The main tools are those of fractal analysis and homogenization: e.g., harmonic extensions obtained by decimation and ad-hoc interpolation and average-value operators on the fibers. Moreover, a crucial role will be played by the choice of the conductivity $a_\varepsilon^n$, which will be singular and discontinuous.\\
After proving the M-convergence of the pre-homogenized functionals, we prove that both the pre-homogenized and the fractal homogenized problems admit unique weak solutions in suitable functional spaces. We point out that the homogenized problem will involve a Venttsel' boundary condition on the fractal boundary. Moreover, from the M-convergence of the functionals we deduce the convergence of the pre-homogenized solutions to the limit fractal homogenized one as $\varepsilon\to 0$ and $n\to+\infty$.

The paper is organized as follows. In Section \ref{preliminari}, we construct the domain $\Omega_\varepsilon^n$ and we introduce the functional setting of this work. In Section \ref{mconv}, we introduce the functionals and we prove that the pre-homogenized functionals M-converge to the fractal functional. Finally, in Section \ref{exunconv}, we prove existence and uniqueness results for both the pre-homogenized and fractal problems and the convergence of the pre-homogenized solutions to the limit fractal one.

\section{Preliminaries}\label{preliminari}
\setcounter{equation}{0}

Throughout this work let $p\geq 2$. We start by introducing the geometry of the problem.\\
Let $T^0$ be the boundary of the triangle of vertices $A=(0,0)$, $B=(1,0)$ e $C=(\frac{1}{2},\frac{\sqrt{3}}{2})$ and let $V_0=\{A,B,C\}$. We construct a bounded (open) domain $\Omega\subset\R^2$ having as fractal boundary $\partial\Omega=K$ the Koch snowflake; we point out that $K$ can be seen as the union of three com-planar Koch curves $K^i$, for $i=1,2,3$. $K^1$ is the uniquely determined self-similar set with respect to the following family of contractive similitudes $\Psi^{(1)}=\{\psi_{1}^{(1)},\psi_{2}^{(1)},\psi_{3}^{(1)},\psi_{4}^{(1)}\}$ (with respect to the same ratio $\frac{1}{3}$) to the side $AB$ of $T^0$:
\begin{center}
$\displaystyle\psi_1^{(1)} (z)=\frac{z}{3},\qquad\psi_2^{(1)} (z)=\frac{z}{3} e^{-i\pi/3}+\frac{1}{3}$,\\[3mm]
$\displaystyle\quad\psi_3^{(1)} (z)=\frac{z}{3} e^{i\pi/3}+\frac{1}{2}-i\frac{\sqrt{3}}{6},
\quad\psi_4^{(1)} (z)=\frac{z+2}{3}.$
\end{center}
We construct in a similar way the curves $K^2$ and $K^3$ as the uniquely determined self-similar sets with respect to suitable families of contractive similitudes $\Psi^{(2)}$ and $\Psi^{(3)}$ respectively. For more details on the construction of $\Omega$ and on the properties of the Koch snowflake, we refer to \cite{freiberg} and \cite{falconer} respectively.

For every $n\in\N$, let $\Omega_n$ be the approximating domain having as boundary $K_n$ the $n$-th approximation of $K$. We point out that every $\Omega_n$ is a bounded polygonal non-convex domain; moreover, the internal angles of the boundary have amplitude equal either to $\frac{\pi}{3}$ or $\frac{4\pi}{3}$. We also denote by $\mathcal{V}^n$ the set of vertices of the polygonal curve $K_n$ and by $\mathcal{V}_{\star}:=\cup_{n\geq 1}\mathcal{V}^n$. We point out that $K=\overline{\mathcal{V}_{\star}}$.

We now introduce the fibers that we will construct around our domain $\Omega_n$. Let $\varepsilon_0=\frac{1}{2}\tan\frac{\pi}{12}$. We denote by $S_1$, $S_2$ and $S_3$ the segments having endpoints $A$ and $B$, $B$ and $C$, and $A$ and $C$ respectively. On every $S_j$ we introduce an $\varepsilon$-neighborhood $\Sigma_{j,\varepsilon}$, for every $0<\varepsilon<\varepsilon_0$ and $j=1,2,3$. More precisely, $\Sigma_{1,\varepsilon}$ is the open polygon having as vertices $A$, $B$, $P_1=(\frac{\varepsilon}{C_1},-\frac{\varepsilon}{2})$ and $P_2=(1-\frac{\varepsilon}{C_1},-\frac{\varepsilon}{2})$, where $C_1=2\varepsilon_0=\tan\frac{\pi}{12}$. We proceed similarly for constructing $\Sigma_{2,\varepsilon}$ and $\Sigma_{3,\varepsilon}$. We point out that every $\Sigma_{j,\varepsilon}$ can be decomposed into the union of a rectangle $\mathcal{R}_{l,\varepsilon}$ and two triangles $\mathcal{T}_{1,l,\varepsilon}$ and $\mathcal{T}_{2,l,\varepsilon}$.\\
We now construct a larger fiber $\Sigma_{j,2\varepsilon}$ of thickness $\varepsilon$, for $j=1,2,3$. For $j=1$, $\Sigma_{1,2\varepsilon}$ is the open polygon of vertices $A$, $B$, $Q_1=(\frac{\varepsilon}{C_1},-\varepsilon)$ and $Q_2=(1-\frac{\varepsilon}{C_1},-\varepsilon)$. We proceed analogously for the construction of $\Sigma_{2,2\varepsilon}$ and $\Sigma_{3,2\varepsilon}$. Obviously $\Sigma_{j,2\varepsilon}$ contains the fiber $\Sigma_{j,\varepsilon}$.\\
We define
\begin{equation*}
\Sigma_{2\varepsilon}=\bigcup_{j=1}^3 \Sigma_{j,2\varepsilon},\qquad \Sigma_{\varepsilon}=\bigcup_{j=1}^3 \Sigma_{j,\varepsilon}.
\end{equation*}
We iterate this procedure on every segment of the pre-fractal curve $K_n$ and we construct two sequences of fibers $\Sigma^n_{\varepsilon}$ and $\Sigma^n_{2\varepsilon}$ respectively.\\
We now introduce a weight $w_\varepsilon^n$ on $\Sigma^n_\varepsilon$. For $i_1,\dots,i_n\in\{1,2,3,4\}$, we denote by $\psi_{i|n}:=\psi_{i_1}\circ\dots\circ\psi_{i_n}$ and, for any measurable set $A$, we set $A^{i|n}:=\psi_{i|n}(A)$. Let $P$ be a point belonging to $\partial\Sigma_{l,\varepsilon}^{i|n}\setminus K_n$ and let $P^\bot$ be its orthogonal projection on $S_l^{i|n}$. Let $(x,y)$ belong to the segment of endpoints $P$ and $P^\bot$. Then
\begin{equation}\label{definizionewn}
w^n_\varepsilon(x,y)=
\begin{cases}
\frac{2^p+C_1^p}{2|P-P^\bot|}\quad &\text{if}\,\,(x,y)\in\mathcal{T}^{i|n}_{j,l,\varepsilon},\,j=1,2,\\[2mm]
\frac{1}{|P-P^\bot|} &\text{if}\,\,(x,y)\in\mathcal{R}^{i|n}_{l,\varepsilon}.
\end{cases}
\end{equation}
The weights $w^n_\varepsilon$ enjoy an important property. We recall that a function $w$ belongs to the Muckenhopt class $A_p$ \cite{muck} if there exists a positive constant $C$ such that for every ball $B\subset\R^2$ it holds that
\begin{equation}\label{condizione muck}
\left(\int_B w\,\de\La_2\right)\cdot\left(\int_B |w|^{-\frac{1}{p-1}}\,\de\La_2\right)^{p-1}\leq C|B|^p.
\end{equation}
As in \cite{MVARMA}, for fixed $\varepsilon$ and $n$, the weights $w_\varepsilon^n$ belong to the class $A_2$, and the constant $C$ in \eqref{condizione muck} can be taken independent from $n$ and $\varepsilon$. Since $p\geq 2$, from H\"older inequality this implies that $w_\varepsilon^n\in A_p$ for fixed $\varepsilon$ and $n$.


We set (see Figure \ref{dominio}) $$\Omega^n_\varepsilon=\Intern(\overline{\Omega_n}\cup\Sigma^n_{2\varepsilon}),$$ where $\Intern(A)$ denotes the interior of a set $A$. We denote by $\Omega^*$ the triangle of vertices $D=(\frac{1}{2},-\frac{\sqrt{3}}{2})$, $E=(\frac{3}{2},\frac{\sqrt{3}}{2})$ and $F=(-\frac{1}{2},\frac{\sqrt{3}}{2})$; we point out that $\Omega^*$ contains $\Omega$ and $\Omega^n_\varepsilon$ for every $n\in\N$ and $0<\varepsilon<\varepsilon_0$. Moreover, let $\Gamma_\varepsilon^n=\partial\Sigma_\varepsilon^n\setminus K_n$.

\begin{figure}[h!]
\centering
\includegraphics[scale=0.75]{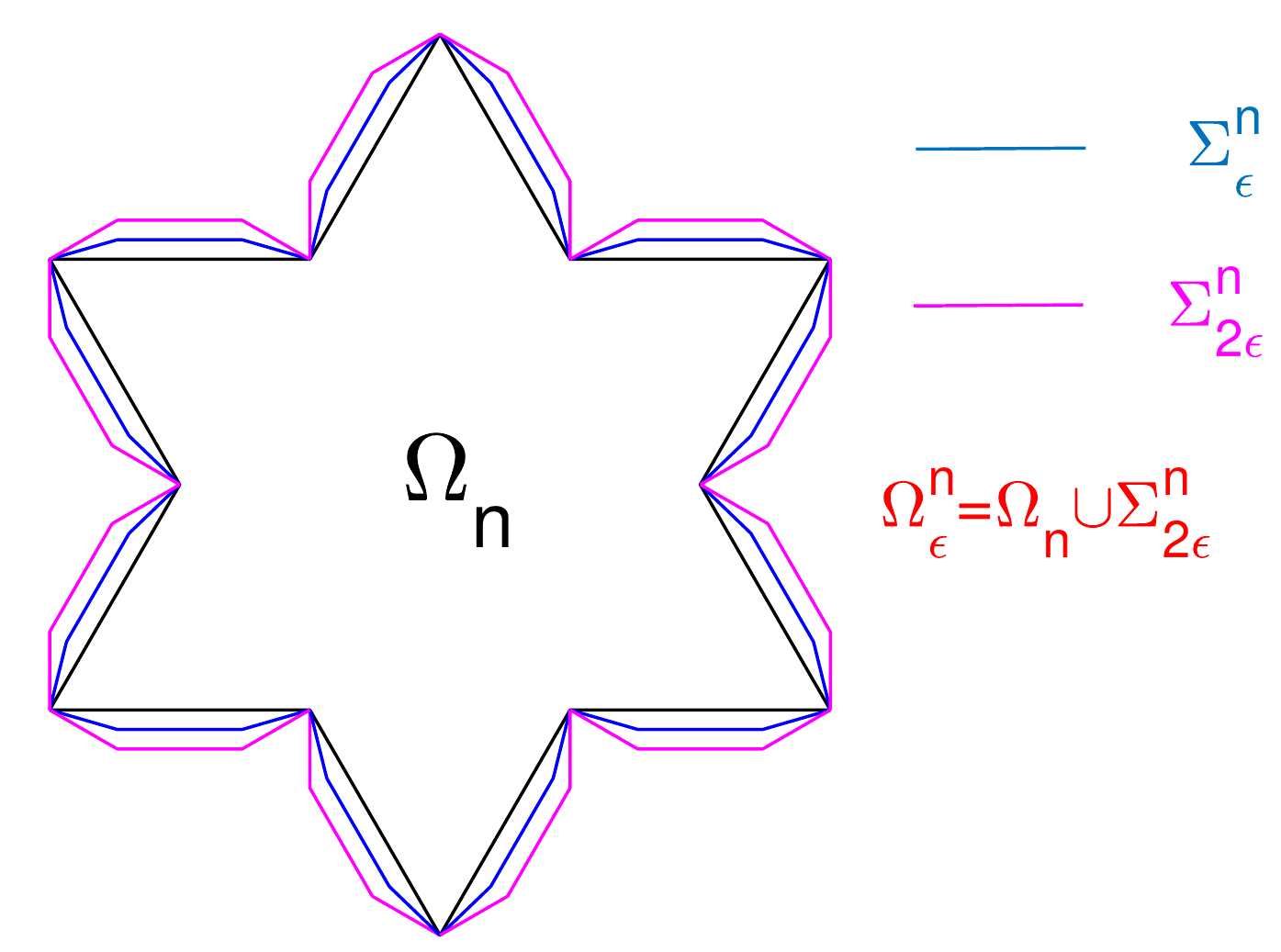}%
\caption{The domain $\Omega_\varepsilon^n$.}%
\label{dominio}
\end{figure}

We can define, in a natural way, a finite Borel measure $\mu$ supported on $K$ by
\begin{equation}\label{eq:1}
\mu:=\mu_1+\mu_2+\mu_3,
\end{equation}
where $\mu_i$ denotes the normalized $d_f$-dimensional Hausdorff measure restricted to $K^i$, $i=1,2,3$, where $d_f=\frac{\log 4}{\log 3}$ is the Hausdorff measure of $K$.\\
For the measure $\mu$ there exist two positive constants $c_1$ and $c_2$ such that
\begin{equation}\label{eq:2}
c_1\,r^d\leq \mu(B(P,r)\cap K)\leq c_2\,r^d,\ \forall\,P\in K,
\end{equation}
where $d=d_f$ and $B(P,r)$ denotes the Euclidean ball of center $P$ and radius $r$. Since $\mu$ is supported on $K$, we can write $\mu(B(P,r))$ in \eqref{eq:2} in place of $\mu(B(P,r)\cap K)$. We note that, in the terminology of \cite{JoWa}, from \eqref{eq:2} it follows that $K$ is a $d_f$-set and the measure $\mu$ is a $d_f$-measure.

\bigskip

By $L^p(\cdot)$ we denote the Lebesgue space with respect to the Lebesgue measure $\de\La_2$ on subsets of $\R^2$, which will be left to
the context whenever that does not create ambiguity. By $L^p(K,\mu)$ we denote the Banach space of $p$-summable functions on $K$ with respect to the invariant measure $\mu$. By $\ell$ we denote the natural arc length coordinate on each edge of $K_n$ and we introduce the coordinates $x=x(\ell)$, $y=y(\ell)$, on every  segment of $K_n$. By $\de\ell$ we denote the one--dimensional measure given by the arc length $\ell$.\\
Let $\G$ be an open set of $\R^2$. By $W^{s,p}(\G)$, where $s\in \R^+$, we denote the (possibly fractional) Sobolev spaces (see \cite{necas}).
Given $\mathcal{S}$ a closed set of $\R^2$, by $C^{0,\alpha}(\mathcal{S})$ we denote the space of H\"older continuous functions on $\mathcal{S}$ of exponent $\alpha$. 

\bigskip

The domains $\Omega_n$ are $(\epsilon, \delta)$ domains with parameters $\epsilon$ and $\delta$ independent of the (increasing) number of sides of $K_n$ (see Lemma 3.3 in \cite{capCPAA}). Thus, by the extension theorem for $(\epsilon, \delta)$ domains due to Jones (Theorem 1 in \cite{Jones}), we obtain the following theorem, which provides an extension operator from $W^{1,p}(\Omega_n)$ to the space ${W^{1,p}(\R^2)}$ whose norm is independent of $n$ (see Theorem 5.7 in \cite{capJMAA}).
\begin{theorem}\label{exte}
There exists a bounded linear extension operator $\Ext_J:\,W^{1,p}(\Omega_n)\to\,W^{1,p}(\R^2)$ such that
\begin{equation}\label{eeea}
\|\Ext_J\, v\|^p_{W^{1,p}(\R^2)} \leq C_J\|v\|^p_{W^{1,p}(\Omega_n)}
\end{equation}
with $C_J$ independent of $n$.
\end{theorem}

\bigskip

We recall a Green formula for Lipschitz domains (see \cite{bregil} and \cite{baiocap}). Let $\mathcal{D}$ be a Lipschitz domain and let $(L^{p'}_{\dive}(\mathcal{D}))^2:=\{w\in (L^{p'}(\mathcal{D}))^2\,:\, \dive w \in L^{p'}(\mathcal{D}) \}$. Then, for every $u,v\in W^{1,p}(\mathcal{D})$ such that $w:=|\nabla{u}|^{p-2}\nabla u\in (L^{p'}_{\dive}(\mathcal{D}))^2$, since $\Delta_p u=\dive(|\nabla u|^{p-2}\nabla u)$, it holds

 $$\int_{\mathcal{D}} |\nabla u|^{p-2}\nabla u\nabla v\,\de\La_2=
 \left\langle\frac{\partial u}{\partial\nu}|\nabla u|^{p-2}, v\right\rangle_{_{W^{-\frac{1}{p'},p'}(\partial\mathcal{D}),W^{\frac{1}{p'},p}(\partial\mathcal{D})}}-\int_{\mathcal{D}} \Delta_p u\,v\,\de\La_2.
$$

\bigskip

We now introduce Besov spaces on the fractal set $K$. From now on, we set $\alpha=1-\frac{2-d_f}{p}$. We define the Besov space on $K$ only for this particular $\alpha$, which is the case of our interest. For a general treatment, see \cite{JoWa}.
\begin{definition}
Let $\mu$ be the measure introduced in \eqref{eq:1} and \eqref{eq:2}. We say that $f \in B_{\alpha}^{p,p}(K)$ if $f \in L^p(K,\mu)$ and
\begin{center}
 $ \Vert f \Vert_{B_{\alpha}^{p,p}(K)}<+\infty$,
\end{center}
where
\begin{equation}\label{eq9}
\Vert f\Vert_{B_{\alpha}^{p,p}(K)}=\Vert f \Vert_{L^p(K,\mu)}+\left(\int\int_{\vert P-P'\vert<1}\frac{\vert f(P)-f(P')\vert^p}{\vert P-P'\vert^{2d_f+p-1}}\de\mu(P)\de\mu(P')\right)^{\frac{1}{p}}
\end{equation}
\end{definition}

We recall a trace theorem.
\begin{theorem}\label{teo traccia Jonsson Wallin} $B_{\alpha}^{p,p}(K)$ is the trace space of $W^{1,p}(\Omega)$ that is:
\begin{enumerate}
\item[1)] There exists a linear and continuous operator $\gamma_0:W^{1,p}(\Omega)\rightarrow B_{\alpha}^{p,p}(K)$.
\item[2)] There exists a linear and continuous operator $\Ext:B_{\alpha}^{p,p}(K)\rightarrow W^{1,p}(\Omega)$ such that $\gamma_0 \circ\Ext$ is the identity operator on $B_{\alpha}^{p,p}(K)$.
\end{enumerate}
\end{theorem}

\medskip
\no For the proof we refer to Theorem 1 of Chapter VII in \cite{JoWa}.

By proceeding as in Theorem 3.7 in \cite{LaVe2}, we can prove the following Green formula for fractal domains. If $w:=|\nabla u|^{p-2}\nabla u\in (L^{p'}_{\dive}(\Omega))^2:=\{w\in (L^{p'}(\Omega))^2\,:\, \dive w \in L^{p'}(\Omega) \}$, since $\Delta_p u=\dive(|\nabla u|^{p-2}\nabla u)$, then for $\alpha=1-\frac{2-d_f}{p}$
 $$\int_{\Omega} |\nabla u|^{p-2}\nabla u\nabla\psi\,\de\La_2=
 \left\langle\frac{\partial u}{\partial\nu}|\nabla u|^{p-2}, \psi\right\rangle_{_{(B^{p,p}_\alpha(K))',B^{p,p}_\alpha(K)}}-\int_{\Omega} \Delta_p u\psi\,\de\La_2
$$
for every $\psi\in W^{1,p}(\Omega)$. Here $(B^{p,p}_\alpha(K))'$ denotes the dual of the Besov space $B^{p,p}_\alpha(K)$ on $K$. This space as shown in \cite{JoWa2} coincides with a subspace of Schwartz distributions $D'(\R^2)$, which are supported on $K$. They are built by means of atomic decomposition. Actually, Jonsson and Wallin proved this result in the general framework of $d$-sets; we refer to \cite{JoWa2} for a complete discussion.

\section{M-convergence of quasi-linear functionals}\label{mconv}
\setcounter{equation}{0}

We introduce for $n\in\N$ and $0<\varepsilon<\varepsilon_0$ the energy functionals which we will consider in order to prove the convergence results.\\
We denote by $W^{1,p}(\Omega_\varepsilon^n,a_\varepsilon^n)$ the Sobolev space of restrictions to $\Omega_\varepsilon^n$ of functions $u$ defined on $\Omega^*$ for which the following norm is finite:
\begin{equation}\label{normapesata}
\|u\|_{W^{1,p}(\Omega_\varepsilon^n,a_\varepsilon^n)}^p:=\|u\|^p_{L^p(\Omega_\varepsilon^n)}+\int_{\Omega_\varepsilon^n}a_\varepsilon^n(x,y)|\nabla u|^p\,\de\La_2.
\end{equation}
Let $\delta_n=\left(\frac{3}{4}\right)^n$. We introduce the following energy functionals defined on $L^2 (\Omega^*)$ 
\begin{equation}\label{funzionalepre}
\Phi_\varepsilon^\enne[u]:=
\begin{cases}
\frac{1}{p}{\displaystyle\int_{\Omega^n_\varepsilon}|u|^{p}\,\de\La_2}+\frac{1}{p}{\displaystyle\int_{\Omega^n_\varepsilon}a_\varepsilon^n(x,y)|\nabla u|^{p}\,\de\La_2}\quad &\text{if}\,\,u|_{\Omega_\varepsilon^n}\in W^{1,p}(\Omega_\varepsilon^n,a_\varepsilon^n),\\
+\infty &\text{if}\,\,u\in L^2 (\Omega^*)\setminus W^{1,p}(\Omega_\varepsilon^n,a_\varepsilon^n),
\end{cases}
\end{equation}

where 
\begin{equation*}
a_\varepsilon^n(x,y)=
\begin{cases}
\delta_n^{1-p} w_\varepsilon^n(x,y)\quad &\text{if}\,\,(x,y)\in\Sigma_\varepsilon^n,\\
1 &\text{if}\,\,(x,y)\in{\Omega^n_\varepsilon}\setminus\overline{\Sigma_{\varepsilon}^n},
\end{cases}
\end{equation*}

and $w_\varepsilon^n(x,y)$ is the weight function defined in \eqref{definizionewn}. 

\begin{prop}\label{dscipref}
$\Phi_\varepsilon^{(n)}$ is a weakly lower semicontinuous, proper and strictly convex functional in $L^2(\Omega^*)$. Moreover, $\Phi_\varepsilon^\enne$ is coercive on $W^{1,p}(\Omega_\varepsilon^n,a_\varepsilon^n)$.
\end{prop}

\begin{proof} It is clear from the definition that $\Phi_\varepsilon^{(n)}$ is proper and strictly convex. The weak lower semicontinuity  follows from the properties of the norms and also the coercivity of $\Phi_\varepsilon^{(n)}$ follows at once.
\end{proof}

\bigskip

We now introduce the fractal energy functional on $\partial\Omega$ (see \cite{Ca02} and \cite{LVSV} for a complete discussion).\\
For $u\colon \mathcal{V}_{\star}\to\mathbb{R}$, we define for $1<p<\infty$ and $n\in\N$ the following quasi-linear discrete energy form:
\begin{equation}\label{e13}
{\mathcal E}_{p}^{(n)}[u]= \frac{4^{(p-1)n}}{p}\sum_{i_1,...,i_n=1}^4
\sum_{\xi,\eta\in V_0} {|u(\psi_{i|n}(\xi))-u(\psi_{i|n}(\eta))|^p}.
\end{equation}

We introduce the nonlinear fractal energy form $\E_p$ with domain $D(\E_p):=\{ u\in C(K): \E_p[u|_{K}]<\infty\}\subset L^p(K,\mu)$ as the following limit: 
\begin{equation}
\label{energyonF}
\E_p[u]=\lim_{n\rightarrow\infty} {\mathcal E}_{p}^{(n)}[u].
\end{equation}

\bigskip

We define the fractal functional:
\begin{equation}\label{funzionalefra}
\Phi[u]:=
\begin{cases}
\frac{1}{p}{\displaystyle\int_{\Omega}|u|^{p}\,\de\La_2}+\frac{1}{p}{\displaystyle\int_{\Omega}|\nabla u|^{p}\,\de\La_2}+\mathcal{E}_p[u]\quad &\text{if}\,\,u|_{\Omega}\in D(\Phi),\\
+\infty &\text{if}\,\,u\in L^2(\Omega^*)\setminus D(\Phi),
\end{cases}
\end{equation}
 with domain
$$D(\Phi):=\left\{u\in W^{1,p}(\Omega)\,:\,
u|_{_{K}}\in D(\mathcal{E}_p)\right\}.$$

We endow $D(\Phi)$ with the following norm:
\begin{equation*}
\|u\|^p_{D(\Phi_p)}:=\int_{\Omega}|u|^{p}\,\de\La_2+\int_{\Omega}|\nabla u|^{p}\,\de\La_2+\mathcal{E}_p[u].
\end{equation*}
From the definition of the $\|\cdot\|_{D(\Phi_p)}$-norm and by proceeding as in \cite[Proposition 2.3]{LVSV}, we can prove the following result.

\begin{prop}\label{dsci}
$\Phi$ is a weakly lower semicontinuous, proper and strictly convex functional in $L^2(\Omega^*)$. Moreover, $\Phi$ is coercive on $D(\Phi)$.
\end{prop}

\bigskip

From now on, we suppose that $\varepsilon\equiv\varepsilon_n$ is a sequence converging to 0 as $n\to+\infty$. Hence, in order to lighten the notation, sometimes we suppress the subscript $\varepsilon$ and we write simply $u_n$ in place of $u_\varepsilon^n$.

We recall the definition of M-convergence adapted to our case. This definition was first introduced by Mosco in \cite{mosco1}; here we recall the definition given in \cite[Definition 2.1.1]{mosco2}, which best fits our aims. 

\begin{definition}\label{def4.1}
Let $H:=L^2(\Omega^*)$. A sequence of proper and convex functionals $\left\{\Phi_\varepsilon^{(n)}\right\}$ defined in $H$ M-converges to
a functional $\Phi$ in $H$ if the following hold:
\begin{itemize}
	\item[\rm a)] for every $\{v_n\}\in H$ weakly converging  to $u\in H$
\begin{equation*}
\underset{n\rightarrow \infty}{\underline\lim}\Phi_\varepsilon^{(n)}
[v_n]\geq \Phi [u].
\end{equation*}
\item[\rm b)] for every $u\in H$ there exists a sequence $\{u_n\}\in H$ strongly converging to $u$ in $H$, such that
\begin{equation}\notag
\underset{n\rightarrow \infty}{\overline\lim}\Phi^{(n)}_\varepsilon[u_n]
\leq \Phi[u].
\end{equation}
\end{itemize}
\end{definition}

We prove a preliminary lemma.
\begin{lemma}\label{lemmapreliminare} For every $u\in D(\Phi)$ there exists a sequence of functions $\hat{u}_m\in C^{0,\beta}(\overline{\Omega^*})\cap W^{1,p}(\Omega^*)$ with $\beta=d_f(1-\frac{1}{p})$ such that $\hat{u}_m\equiv u|_K$ on $K$ and it converges strongly to $u$ in $L^2(\Omega^*)$. Moreover, $\hat{u}_m\xrightarrow[m\to+\infty]{}u$ strongly in $W^{1,p}(\Omega)$.
\end{lemma}

\begin{proof}
Let $u\in D(\Phi)$, so that in particular $u|_{K}\in D(\E_p)$. From \cite[Theorem 3, page 155]{JoWa} there exists an extension operator from $W^{1,p}(\Omega)$ to $W^{1,p}(\Omega^*)$; with an abuse of notation, we still denote the extension of $u\in D(\Phi)$ to $W^{1,p}(\Omega^*)$ by $u$. From Theorem 4.1 in \cite{caplanconvex}, this latter space coincides with the space $\Lip_{d_f,p,\infty}(K)$ (for the definition see Jonsson \cite{jonssonlip}). We note that the space $\Lip_{d_f,p,\infty}(K)$ is a subspace of the Besov space $B^{p,\infty}_{d_f}(K)$; hence we can extend $u|_K$ to a function $\hat{u}$ defined on $\R^2$ and belonging to the space $B^{p,\infty}_{d_f+\frac{2-d_f}{p}}(\R^2)$. From the properties of Besov spaces we have that $B^{p,\infty}_{d_f+\frac{2-d_f}{p}}(\R^2)\subset B^{p,p}_{d_f+\frac{2-d_f}{p}-\epsilon}(\R^2)$ for every $\epsilon>0$. The latter space coincides with the usual Sobolev space $W^{d_f+\frac{2-d_f}{p}-\epsilon,p}(\R^2)$. By embedding properties of Besov spaces, we get that $\hat{u}\in C^{0,\beta}(\R^2)$ for $\beta=d_f(1-\frac{1}{p})$ \cite[page 5]{JoWa}. In particular, the trace of $\hat{u}$ on $K$ belongs to $C^{0,\beta}(K)$ and we identify $u|_K$ with this trace.\\
Since in particular $u\in W^{1,p}(\Omega^*)$, we have that $w:=\hat{u}-u$ belongs to $W^{1,p}(\Omega^*)$ and has zero trace on $K$. From the density of $C^1(\overline{\Omega^*})$ in $W^{1,p}(\Omega^*)$, there exists a sequence $\{u_m^*\}$, with $u_m^*\in C^1(\overline{\Omega^*})$ and $u^*_m\equiv 0$ on $K$, which converges strongly to $w=\hat{u}-u$ in $W^{1,p}(\Omega^*)$ for $m\to+\infty$. For every $m$ we define $\hat{u}_m:=-u^*_m+\hat{u}$. Hence, the sequence $\{\hat{u}_m\}$ converges to $u$ strongly in $L^2(\Omega^*)$ and $W^{1,p}(\Omega^*)$ and $\E_p[\hat{u}_m]=\E_p[u|_K]$. Finally, the strong convergence of $\{\hat{u}_m\}$ to $u$ in $W^{1,p}(\Omega^*)$ implies the strong convergence on $W^{1,p}(\Omega)$, thus concluding the proof.
\end{proof}

\medskip

Before focusing on the main convergence result, we need to introduce some tools. As in \cite{mosvivthin}, we introduce a family of nested and regular triangulations $\mathcal{T}_n$ of $\Omega^*$ such that at every iteration $n$ the vertices of $K_n$ are also nodes of the triangulation. We denote by $\mathcal{P}_n$ the set of vertices of all the triangles of $\mathcal{T}_n$ and by $\mathcal{S}_n$  the space of continuous functions on $\Omega^*$ affine on every triangle of $\mathcal{T}_n$; we point out that for every $n$ it holds that $\mathcal{P}_n\subset\mathcal{P}_{n+1}$ and $\mathcal{S}_n\subset\mathcal{S}_{n+1}$.\\
By suitably adapting the proofs of Propositions 4.1 and 4.2 of \cite{mosvivthin} (see also Theorems 1 and 2 in \cite{emily}), we get the following result.
\begin{prop}\label{prop4.2}
For every $u\in C^{0,\beta}(\overline{\Omega^*})\cap W^{1,p}(\Omega^*)$ such that $u|_K\in D(\mathcal{E}_p)$ there exists a sequence of piecewise affine functions $I_nu$ interpolating $u$ in the nodes of $\mathcal{V}^n$ which converges to $u$ in $W^{1,p}(\Omega^*)$ and satisfies the estimate
\begin{equation}\label{stima interpolata}
|I_nu(P)-I_nu(Q)|\leq C|P-Q|^\beta\quad\text{for every}\quad P,Q\in\mathcal{P}_n,
\end{equation}
where the constant $C$ is independent from $n$. Moreover, $I_nu(P)=u(P)\quad\forall\,P\in\mathcal{V}^n$.
\end{prop}

We now prove the M-convergence result.

\begin{theorem}\label{mconvergenza}
Let $\delta_n=(3^{1-d_f})^n=\left(\frac{3}{4}\right)^n$ and $\varepsilon\equiv\varepsilon_n$ be a sequence converging to 0 as $n\to+\infty$. Let $\Phi$ and $\Phi_\varepsilon^\enne$ be defined as in \eqref{funzionalefra} and \eqref{funzionalepre} respectively. Then $\Phi_\varepsilon^{(n)}$ M-converges to the functional $\Phi$ as $n\to+\infty$.
\end{theorem}

\begin{proof}
{\bf 1) Limsup condition}: We can suppose that $u\in D(\Phi)$ since if $u\notin D(\Phi)$ then $\Phi[u] = +\infty $ and the thesis would trivially hold.\\
We first assume that $u\in C^{0,\beta}(\overline{\Omega^*})$, where we recall that $\beta=d_f(1-\frac{1}{p})$. We have to construct a sequence $u_n$ strongly converging in $L^2(\Omega^*)$ to $u$. 

We define the following operator
\begin{center}
$G_\varepsilon\colon C^{0,\beta}(\overline{\Omega^*})\cap W^{1,p}(\Omega^*)\to C^{0,\beta}(\overline{\Omega^*})\cap W^{1,p}(\Omega^*)$
\end{center}
which acts on $\Omega^n_\varepsilon$. Let $(x_l,y_l)$ denote the orthogonal projection of $(x,y)\in\Sigma_{l,2\varepsilon}$ on $S_l$, for $l=1,2,3$. Then, for every point $P=(x,y)$ of $\Sigma_{l,2\varepsilon}\setminus\Sigma_{l,\varepsilon}$, we set $\hat{P}_l=(\hat{x}_l,\hat{y}_l)\in\partial\Sigma_{l,\varepsilon}$ and $\tilde{Q}_l=(\tilde{x}_l,\tilde{y}_l)\in\partial\Sigma_{l,2\varepsilon}$ to be the intersections between the straight line orthogonal to $S_l$ at $(x_l,y_l)$ and $\partial\Sigma_{l,\varepsilon}\setminus S_l$ and $\partial\Sigma_{l,2\varepsilon}\setminus S_l$ respectively (see Figure \ref{zoom}).

\begin{figure}
\centering
\includegraphics[scale=0.8]{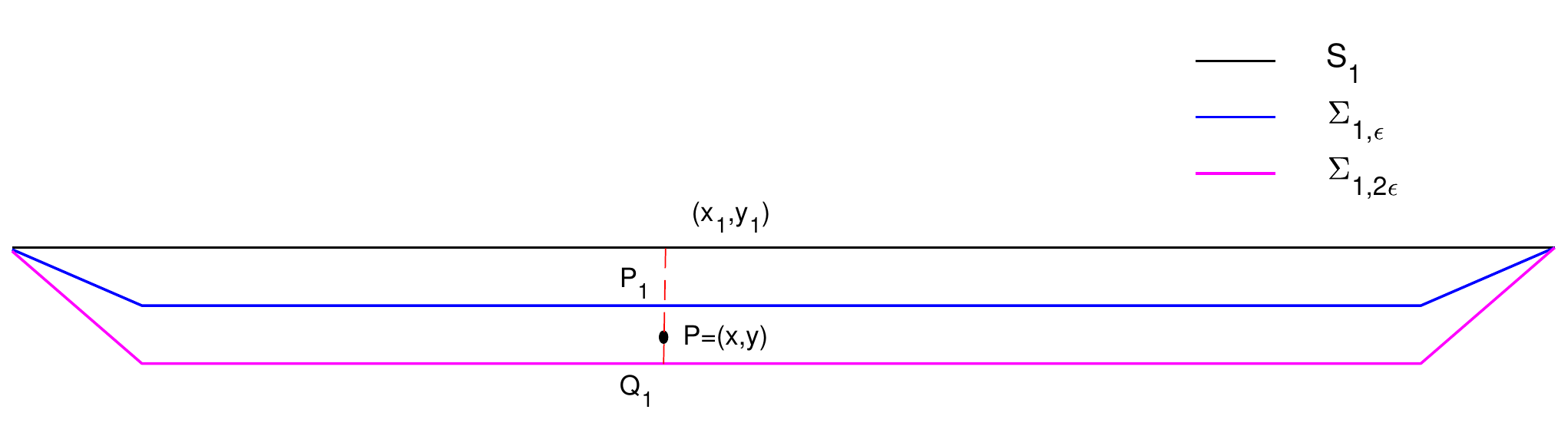}%
\caption{A possible example of $\hat{P}_l$ and $\tilde{Q}_l$.}%
\label{zoom}%
\end{figure}

Hence, for $g\in C^{0,\beta}(\overline{\Omega^*})$, we define 
\begin{equation}\label{definizioneg}
G_\varepsilon(g(x,y))=
\begin{cases}
g(x,y) &\text{if}\,\,(x,y)\in\Omega^*\setminus\Sigma_{l,2\varepsilon},\\[2mm]
I_0g(x_l,y_l) &\text{if}\,\,(x,y)\in\overline{\Sigma_{l,\varepsilon}},\\[2mm]
I_0g(x_l,y_l)t_l+g(\tilde{Q}_l)(1-t_l)\quad &\text{if}\,\,(x,y)\in\overline{\Sigma_{l,2\varepsilon}}\setminus\overline{\Sigma_{l,\varepsilon}},
\end{cases}
\end{equation}
where
\begin{equation*}
t_l=\frac{|\tilde{y}_l-y|+|\tilde{x}_l-x|}{|\tilde{y}_l-\hat{y}_l|+|\tilde{x}_l-\hat{x}_l|}.
\end{equation*}
We now set 
\begin{equation}\label{definizioneun}
u_n=
\begin{cases}
I_nu &\text{if}\,\,(x,y)\in\Omega^*\setminus\Sigma_{2\varepsilon}^n,\\[2mm]
G_\varepsilon(I_n u\circ\psi_{i|n})\circ\psi_{i|n}^{-1}\quad &\text{if}\,\,(x,y)\in\Sigma_{2\varepsilon}^n,
\end{cases}
\end{equation}
where $I_nu$ is the interpolating function given by Proposition \ref{prop4.2}.

We point out that, since $u$ is H\"older continuous on $\Omega^*$, from the definition of $u_n$ it follows that
\begin{equation}\label{limitatezza}
\max_{\Sigma_{2\varepsilon}^n}|u_n|\leq\|u\|_{L^\infty(\Omega^*)}.
\end{equation}
Hence
\begin{equation*}
\begin{split}
&\|u_n-u\|^2_{L^2(\Omega^*)}=\|I_nu-u\|^2_{L^2(\Omega^*\setminus\Sigma_{2\varepsilon}^n)}+\int_{\Sigma_{2\varepsilon}^n}|u_n-u|^2\,\de\La_2\\
&\leq\|I_nu-u\|^2_{L^2(\Omega^*)}+C\left(\max_{\Sigma_{2\varepsilon}^n}|u_n|^2+\max_{\Sigma_{2\varepsilon}^n}|u|^2\right)|\Sigma_{2\varepsilon}^n|;
\end{split}
\end{equation*}
since $|\Sigma_{2\varepsilon}^n|\to 0$ as $n\to+\infty$, from \eqref{limitatezza} and Proposition \ref{prop4.2} we get that $u_n$ converges to $u$ strongly in $L^2(\Omega^*)$.


According to the definition of $u_n$ and $a_\varepsilon^n$, we split the functional $\Phi_\varepsilon^\enne$ as follows:
\begin{equation*}
\begin{split}
\Phi_\varepsilon^\enne[u_n]&=\frac{1}{p}\int_{\Omega_n}|I_n u|^{p}\,\de\La_2+\frac{1}{p}\int_{\Sigma^n_{2\varepsilon}\setminus\Sigma_{\varepsilon}^n}|u_n|^{p}\,\de\La_2+\frac{1}{p}\int_{\Sigma^n_\varepsilon}|u_n|^{p}\,\de\La_2\\[2mm]
&+\frac{1}{p}\int_{\Omega_n}|\nabla I_n u|^{p}\,\de\La_2+\frac{1}{p}\int_{\Sigma^n_{2\varepsilon}\setminus\Sigma_{\varepsilon}^n}|\nabla u_n|^{p}\,\de\La_2+\frac{\delta_n^{1-p}}{p}\int_{\Sigma^n_\varepsilon}w_\varepsilon^n|\nabla u_n|^{p}\,\de\La_2.
\end{split}
\end{equation*}

Since the domain $\Omega_n$ tends to the fractal domain $\Omega$ as $n\to+\infty$, from the properties of the interpolating functions we get that
\begin{equation*}
\lim_{n\to+\infty}\frac{1}{p}\int_{\Omega_n}|I_n u|^{p}\,\de\La_2=\frac{1}{p}\int_{\Omega}|u|^{p}\,\de\La_2\quad\text{and}\quad\lim_{n\to+\infty}\frac{1}{p}\int_{\Omega_n}|\nabla I_nu|^{p}\,\de\La_2=\frac{1}{p}\int_{\Omega}|\nabla u|^{p}\,\de\La_2
\end{equation*}
We should prove that
\begin{equation}\label{condizione3}
\lim_{n\to+\infty}\int_{\Sigma^n_{2\varepsilon}\setminus\Sigma_{\varepsilon}^n}|u_n|^{p}\,\de\La_2=\lim_{n\to+\infty}\int_{\Sigma^n_\varepsilon}|u_n|^{p}\,\de\La_2=0,
\end{equation}
\begin{equation}\label{condizione1}
\lim_{n\to+\infty}\int_{\Sigma^n_{2\varepsilon}\setminus\Sigma_{\varepsilon}^n}|\nabla u_n|^{p}\,\de\La_2=0
\end{equation}
and
\begin{equation}\label{condizione2}
{\underset{n\to+\infty}{\overline\lim}}\frac{\delta_n^{1-p}}{p}\int_{\Sigma^n_\varepsilon}w_\varepsilon^n|\nabla u_n|^{p}\,\de\La_2\leq\E_p[u].
\end{equation}

We begin by proving \eqref{condizione3} and \eqref{condizione1}. We point out that by definition
\begin{equation}\notag
\Sigma^n_\varepsilon=\bigcup_{i|n}\Sigma^{i|n}_\varepsilon\quad\text{and}\quad\Sigma^n_{2\varepsilon}=\bigcup_{i|n}\Sigma^{i|n}_{2\varepsilon}.
\end{equation}
Moreover, we can decompose $\Sigma^n_{2\varepsilon}\setminus\Sigma_{\varepsilon}^n$ in the sum of three rectangles and six triangles; hence
\begin{equation}\label{decomposizione1}
\int_{\Sigma^n_{2\varepsilon}\setminus\Sigma_{\varepsilon}^n}=\bigcup_{i|n}\int_{\Sigma^{i|n}_{2\varepsilon}\setminus\Sigma_{\varepsilon}^{i|n}}=\bigcup_{i|n}\left(\sum_{l=1}^3\int_{\psi_{i|n}(R_l)}+\sum_{l=1}^3\sum_{j=1}^2\int_{\psi_{i|n}(T_{l,j})}\right).
\end{equation}
We start with the integrals on the rectangles. Since the computations are similar, we explicitly compute only the integral on the rectangle $R_1$, which has as vertices $P_1$, $P_2$, $Q_1$ and $Q_2$. We enumerate the other rectangles starting from $R_1$ going counterclockwise.\\
We set $g(x,y):=(I_n u\circ\psi_{i|n})(x,y)$. On $R_1$, from Proposition \ref{prop4.2}, the function $G_\varepsilon(g(x,y))$ is defined as follows:%
\begin{center}
$\displaystyle G_\varepsilon(g(x,y))=-\frac{2y}{\sqrt{3}}\left(u(\psi_{i|n}(A))+u(\psi_{i|n}(B))-2u(\psi_{i|n}(D))\right)+2\left(u(\psi_{i|n}(A))(1-x)+u(\psi_{i|n}(B))x\right)-\left(u(\psi_{i|n}(A))\left(1-x+\frac{\varepsilon}{\sqrt{3}}\right)+u(\psi_{i|n}(B))\left(x+\frac{\varepsilon}{\sqrt{3}}\right)-u(\psi_{i|n}(D))\frac{2\varepsilon}{\sqrt{3}}\right)$.
\end{center}

We recall that $u\in C^{0,\beta}(\overline{\Omega^*})$ by hypothesis. By a change of variables and integrating, we get
\begin{equation}\label{stimarett per u}
\!\!\!\!\!\!\!\!\!\!\!\!\!\!\!\int_{\psi_{i|n}(R_1)}|G_\varepsilon(g)|^p\,\de\La_2\leq \Theta^1_p \left\{3^{-2n}3^{-n\beta p}\left(1-\frac{2\varepsilon}{C_1}\right)\varepsilon^{p+1}+3^{-2n}3^{-n\beta p}\frac{\varepsilon}{2}\left[\left(1-\frac{\varepsilon}{C_1}\right)^{p+1}-\frac{\varepsilon^{p+1}}{C_1^{p+1}}\right]+3^{-2n}|R_1|\right\},
\end{equation}
where $\Theta^1_p$ is a suitable positive constant which depends in particular on $p$ and on the H\"older constant of $u$.

We now compute $(G_\varepsilon(g))_x$ and $(G_\varepsilon(g))_y$; since $u\in C^{0,\beta}(\overline{\Omega^*})$, we have that both derivatives are bounded by constants not depending on $\varepsilon$. Hence we get
\begin{equation}\label{stimarettangolo}
\int_{\psi_{i|n}(R_1)}|\nabla G_\varepsilon(g)|^p\,\de\La_2\leq 3^{n(p-2)}C_H^p C^*_p 3^{-n\beta p} |R_1|=3^{n(p-2)}C_H^p C^*_p 3^{-n\beta p}\left(1-\frac{2\varepsilon}{C_1}\right)\frac{\varepsilon}{2},
\end{equation}
where $C^*_p$ is a constant depending on $p$ and $C_H$ is the H\"older constant of $u$.\\
We pass now to the integrals on the triangles. As above, we explicitly compute the integral on the triangle $T_{1,1}$ of vertices $A$, $P_1$ and $Q_1$; the other triangles can be computed similarly.\\
On $T_1$, the function $G_\varepsilon(g(x,y))$ is defined as follows:
\begin{center}
$\displaystyle G_\varepsilon(g(x,y))=-\frac{2y}{\sqrt{3}}\left(u(\psi_{i|n}(A))+u(\psi_{i|n}(B))-2u(\psi_{i|n}(D))\right)+2\left(u(\psi_{i|n}(A))(1-x)+u(\psi_{i|n}(B))x\right)-\left(u(\psi_{i|n}(A))\left(1-x+\frac{C_1x}{\sqrt{3}}\right)+u(\psi_{i|n}(B))\left(x+\frac{C_1x}{\sqrt{3}}\right)-u(\psi_{i|n}(D))\frac{2C_1x}{\sqrt{3}}\right)$.
\end{center}

By changing the variables and integrating, we get
\begin{equation}\label{stimatriang per u}
\int_{\psi_{i|n}(T_{1,1})}|G_\varepsilon(g)|^p\,\de\La_2\leq \Theta^2_p \left(3^{-2n}3^{-n\beta p}\varepsilon^{p+2}+3^{-2n}\varepsilon^2\right),
\end{equation}
where $\Theta^2_p$ is a suitable positive constant which depends in particular on $p$ and on the H\"older constant of $u$.

As before, $(G_\varepsilon(g))_x$ and $(G_\varepsilon(g))_y$ are bounded by constant not depending on $\varepsilon$. Hence we get
\begin{equation}\label{stimatriangolo}
\int_{\psi_{i|n}(T_{1,1})}|\nabla G_\varepsilon(g)|^p\,\de\La_2\leq 3^{n(p-2)}C_H^p \hat{C}_p 3^{-n\beta p} \int_0^\frac{\varepsilon}{C_1}\frac{C_1x}{2}\,\de x=3^{n(p-2)}C_H^p \hat C_p 3^{-n\beta p}\frac{\varepsilon^2}{4C_1}.
\end{equation}

From \eqref{stimarett per u}, \eqref{stimatriang per u} and \eqref{decomposizione1} we get that there exists a suitable positive constant $\Theta_p$ such that
\begin{equation*}
\int_{\Sigma^n_{2\varepsilon}\setminus\Sigma_{\varepsilon}^n}|u_n|^p\,\de\La_2=\sum_{i|n}\int_{\Sigma^{i|n}_{2\varepsilon}\setminus\Sigma_{\varepsilon}^{i|n}}|u_n|^p\,\de\La_2\leq\sum_{i|n} \Theta_p 3^{-2n}(3^{-n\beta p}+1)\varepsilon=\Theta_p \varepsilon\left(\frac{3^{-2n}}{4^{n(p-2)}}+\frac{4^n}{3^{2n}}\right),
\end{equation*}
thus proving the first part of \eqref{condizione3}. The second assertion of \eqref{condizione3}, namely that
$$\lim_{n\to+\infty}\int_{\Sigma^n_\varepsilon}|u_n|^{p}\,\de\La_2=0,$$
 can be proved in a similar way.

From \eqref{stimarettangolo}, \eqref{stimatriangolo} and \eqref{decomposizione1} we get
\begin{equation*}
\int_{\Sigma^n_{2\varepsilon}\setminus\Sigma_{\varepsilon}^n}|\nabla u_n|^p\,\de\La_2=\sum_{i|n}\int_{\Sigma^{i|n}_{2\varepsilon}\setminus\Sigma_{\varepsilon}^{i|n}}|\nabla u_n|^p\,\de\La_2\leq\sum_{i|n} 3^{n(p-2)}C_H^p \tilde{C}_p\varepsilon 3^{-n\beta p}=C_H^p \tilde{C}_p\varepsilon\frac{3^{n(p-2)}}{4^{n(p-2)}},
\end{equation*}
thus proving \eqref{condizione1}.\\

We have now to prove \eqref{condizione2}. By proceeding as above, we split the integrals on $\Sigma_\varepsilon^n$ in the sum of integrals on triangles and rectangles:
\begin{equation*}
\frac{\delta_n^{1-p}}{p}\int_{\Sigma^n_\varepsilon}w_\varepsilon^n|\nabla u_n|^{p}\,\de\La_2=\frac{\delta_n^{1-p}}{p}\bigcup_{i|n}\int_{\Sigma_{\varepsilon}^{i|n}}w_\varepsilon^n|\nabla u_n|^{p}\,\de\La_2=\frac{\delta_n^{1-p}}{p}\bigcup_{i|n}\left(\sum_{l=1}^3\int_{\psi_{i|n}(R_l)}+\sum_{l=1}^3\sum_{j=1}^2\int_{\psi_{i|n}(T_{l,j})}\right).
\end{equation*}

As before, we perform the calculations on $\Sigma_{1,\varepsilon}$. With an abuse of notation, let $R_1$ be the rectangle of vertices $P_1$, $P_2$, $P_3=(\frac{\varepsilon}{C_1},0)$ and $P_4=(1-\frac{\varepsilon}{C_1},0)$. We point out that on $\psi_{i|n}(\Sigma_{1,\varepsilon})$, $w^n_\varepsilon=3^n\,\ell_\varepsilon^{-1}(x)$ with
\begin{equation*}
\ell_\varepsilon(x)=
\begin{cases}
\displaystyle\frac{C_1x}{2^p+C_1^p} &0<x<\frac{\varepsilon}{C_1},\\[3mm]
\displaystyle\frac{\varepsilon}{2} &\frac{\varepsilon}{C_1}<x<1-\frac{\varepsilon}{C_1},\\[3mm]
\displaystyle\frac{C_1-C_1x}{2^p+C_1^p}\quad &1-\frac{\varepsilon}{C_1}<x<1.
\end{cases}
\end{equation*}

If $G_\varepsilon(g(x,y))$ is defined as in the above case of the rectangle, recalling the definition of $\delta_n$ we have that
\begin{center}
$\displaystyle\frac{\delta_n^{1-p}}{p}\int_{\psi_{i|n}(R_1)}w_\varepsilon^n|\nabla u_n|^{p}\,\de\La_2=\frac{\delta_n^{1-p}}{p}\frac{2\cdot 3^n}{\varepsilon} 3^{n(p-2)}\int_\frac{\varepsilon}{C_1}^{1-\frac{\varepsilon}{C_1}}\int^0_{-\frac{\varepsilon}{2}}\left|u(\psi_{i|n}(A))-u(\psi_{i|n}(B))\right|^p\,\de y\,\de x
=\frac{4^{n(p-1)}}{p}\frac{2}{\varepsilon}\left(1-\frac{2\varepsilon}{C_1}\right)\frac{\varepsilon}{2}\left|u(\psi_{i|n}(A))-u(\psi_{i|n}(B))\right|^p$.
\end{center}

Again with an abuse of notation, now let $T_{1,1}$ be the triangle of vertices $A$, $P_1$ and $P_3$. Hence we can compute the integral on $T_{1,1}$ as follows:
\begin{center}
$\displaystyle\frac{\delta_n^{1-p}}{p}\int_{\psi_{i|n}(T_{1,1})}w_\varepsilon^n|\nabla u_n|^{p}\,\de\La_2=\frac{\delta_n^{1-p}}{p}\frac{3^n(2^p+C_1^p)}{C_1} 3^{n(p-2)}\int_0^{\frac{\varepsilon}{C_1}}\int^0_{-\frac{C_1x}{2}}\frac{\left|u(\psi_{i|n}(A))-u(\psi_{i|n}(B))\right|^p}{x}\,\de y\,\de x=\frac{4^{n(p-1)}}{p}\frac{2^p+C_1^p}{2}\frac{\varepsilon}{C_1}\left|u(\psi_{i|n}(A))-u(\psi_{i|n}(B))\right|^p$.
\end{center}
Proceeding counterclockwise from $T_{1,1}$, we denote by $T_{1,2}$ the triangle of vertices $B$, $P_2$ and $P_4$. Then we have that
\begin{center}
$\displaystyle\frac{\delta_n^{1-p}}{p}\int_{\psi_{i|n}(T_{1,2})}w_\varepsilon^n|\nabla u_n|^{p}\,\de\La_2=\frac{\delta_n^{1-p}}{p}\frac{3^n(2^p+C_1^p)}{2} 3^{n(p-2)}\int^1_{1-\frac{\varepsilon}{C_1}}\int^0_{\frac{C_1(x-1)}{2}}\frac{\left|u(\psi_{i|n}(A))-u(\psi_{i|n}(B))\right|^p}{C_1-C_1x}\,\de y\,\de x=\frac{4^{n(p-1)}}{p}\frac{2^p+C_1^p}{2}\frac{\varepsilon}{C_1}\left|u(\psi_{i|n}(A))-u(\psi_{i|n}(B))\right|^p$.
\end{center}
Summing the above integrals, we get
\begin{center}
$\displaystyle\frac{\delta_n^{1-p}}{p}\int_{\psi_{i|n}(\Sigma_{1,\varepsilon})}w_\varepsilon^n|\nabla u_n|^{p}\,\de\La_2=\frac{4^{n(p-1)}}{p}\left(1-\frac{2\varepsilon}{C_1}+(2^p+C_1^p)\frac{\varepsilon}{C_1}\right)\left|u(\psi_{i|n}(A))-u(\psi_{i|n}(B))\right|^p$.
\end{center}
By applying the above reasoning to the other integrals, setting $C_p:=2^p+C_1^p-2$ we have
\begin{align*}
&\frac{\delta_n^{1-p}}{p}\int_{\Sigma_{\varepsilon}^n}w_\varepsilon^n|\nabla u_n|^{p}\,\de\La_2=\left(1+\frac{\varepsilon}{C_1}C_p\right)\frac{4^{n(p-1)}}{p}\sum_{i|n}\left(\left|u(\psi_{i|n}(A))-u(\psi_{i|n}(B))\right|^p\right.\\[2mm]
&\left.+\left|u(\psi_{i|n}(B))-u(\psi_{i|n}(C))\right|^p+\left|u(\psi_{i|n}(A))-u(\psi_{i|n}(C))\right|^p\right).
\end{align*}
Taking into account the definition of $\E_p$, we get that
\begin{equation*}
\frac{\delta_n^{1-p}}{p}\int_{\Sigma_{\varepsilon}^n}w_\varepsilon^n|\nabla u_n|^{p}\,\de\La_2\leq\left(1+\frac{\varepsilon}{C_1}C_p\right) \E_p[u],
\end{equation*}
and by taking the limit for $n\to+\infty$ we get \eqref{condizione2}, since $\varepsilon\equiv\varepsilon_n\to 0$.

We now remove the assumption $u\in C^{0,\beta}(\overline{\Omega^*})$. From Lemma \ref{lemmapreliminare}, for every $u\in D(\Phi)$ there exists a sequence of functions $\hat{u}_m\in C^{0,\beta}(\overline{\Omega^*})$ which converges strongly to $u$ both in $L^2(\Omega^*)$ and in $W^{1,p}(\Omega)$ for $m\to+\infty$ and such that $\hat{u}_m=u$ on $K$. From the first part of the proof, for every function $\hat{u}_m\in C^{0,\beta}(\overline{\Omega^*})$ there exists a sequence of functions $\hat{u}_{m,n}\in L^2(\Omega^*)$ such that 
\begin{equation}\label{approssimante}
\lim_{n\to+\infty}\|\hat{u}_{m,n}-\hat{u}_{m}\|_{L^2(\Omega^*)}=0\quad\text{and}\quad\underset{n\rightarrow \infty}{\overline\lim}\Phi_\varepsilon^\enne[\hat{u}_{m,n}]\leq\Phi[\hat{u}_m].
\end{equation}
From the properties of $\hat{u}_m$ and \eqref{approssimante} we get
\begin{equation*}
\begin{split}
\Phi[u]&=\frac{1}{p}\int_{\Omega}|u|^{p}\,\de\La_2+\frac{1}{p}\int_{\Omega}|\nabla u|^{p}\,\de\La_2+\mathcal{E}_p[u]=\lim_{m\to+\infty}\left(\frac{1}{p}\int_{\Omega}|\hat u_m|^{p}\,\de\La_2+\frac{1}{p}{\displaystyle\int_{\Omega}|\nabla\hat{u}_m|^{p}\,\de\La_2}+\mathcal{E}_p[\hat{u}_m]\right)\\[3mm]
&=\lim_{m\to+\infty}\Phi[\hat{u}_m]\geq\lim_{m\to+\infty}\left(\underset{n\rightarrow \infty}{\overline\lim}\Phi_\varepsilon^\enne[\hat{u}_{m,n}]\right).
\end{split}
\end{equation*}
In addition to that, by direct calculations we have that
\begin{equation*}
\lim_{m\to+\infty}\left(\lim_{n\to+\infty}\|\hat{u}_{m,n}-u\|_{L^2(\Omega^*)}\right)=0.
\end{equation*}
We now use the diagonal formula given in Corollary 1.16 in \cite{Att}. There exists a strictly increasing mapping $n\mapsto m(n)$ such that $m(n)\to+\infty$ when $n\to+\infty$ such that, by putting $u_n=u_{m(n),n}$
\begin{equation*}
\lim_{n\to+\infty}\|u_{n}-u\|_{L^2(\Omega^*)}=0\quad\text{and}\quad\underset{n\rightarrow \infty}{\overline\lim}\Phi_\varepsilon^\enne[u_n]\leq\Phi[u],
\end{equation*}
thus proving the thesis.


{\bf 2) Liminf condition}: Let $v_n$ be a weakly converging sequence to $u$ in $L^2(\Omega^*)$.
We can suppose that $v_n\in W^{1,p}(\Omega_\varepsilon^n,a_\varepsilon^n)$ and $$\displaystyle\underset{n\rightarrow \infty}{\underline\lim}\Phi_\varepsilon^{(n)}[v_n]<\infty$$
(otherwise the thesis follows trivially). We first assume that $v_n\in C^{1}(\overline{\Omega^n_\varepsilon})$ for every $n\in\N$.
Then there exists a constant $c$ independent of $n$ and $\varepsilon$ such that
\begin{equation}\label{stimap}
\Phi_\varepsilon^\enne[v_n]=\frac{1}{p}\int_{\Omega^n_\varepsilon}|v_n|^{p}\,\de\La_2+\frac{1}{p}\int_{\Omega_n}|\nabla v_n|^{p}\,\de\La_2+\frac{1}{p}\int_{\Sigma_{2\varepsilon}^n} a_\varepsilon^n|\nabla v_n|^{p}\,\de\La_2\leq c.
\end{equation}
We point out that, from \eqref{stimap}, we can suppose that $\|v_n\|_{W^{1,p}(\Omega_n)}\leq C$, where $C$ is independent of $n$. For every $n\in\N$ from Theorem \ref{exte} there exists a bounded linear operator $\Ext_J\colon W^{1,p}(\Omega_n)\to W^{1,p}(\R^2)$ such that
  $$ \|\Ext_J\,v_n\|_{W^{1,p}(\R^2)}\leq C\,\|v_n\|_{W^{1,p}({\Omega_n} )} \leq C,$$
with $C$ independent of $n$.\\
We denote by $\hat v_n=\Ext_J\,v_n|_{\Omega}$. Then $\hat v_n\in W^{1,p}(\Omega)$ and $\|\hat v_n\|_{W^{1,p}(\Omega)}\leq C$; hence there exists a subsequence, still denoted by $\hat v_n$, weakly converging to $\hat v$ in $W^{1,p}(\Omega)$. We point out that $\hat v_n$ strongly converges to $\hat v$ in $L^p(\Omega)$ and also in $L^2(\Omega)$ since $p\geq 2$. By proceeding as in the proof of condition a) of \cite[Theorem 3.5]{CPAA}, we prove that $\hat v= u$ a.e. on $\Omega$, hence in particular $\hat{v}_n$ weakly converges to $u$ in $W^{1,p}(\Omega)$.\\
As in \cite{CPAA}, this implies that $\chi_{\Omega_n}\nabla v_n$ weakly converges to $\nabla u$ in $L^p(\Omega)$ and, from the lower semicontinuity of the norm, we get that
\begin{equation*}
\underset{n\rightarrow \infty}{\underline\lim}\int_{\Omega_n}|\nabla v_n|^p\,\de\La_2 \geq \int_\Omega |\nabla u|^p\,\de\La_2.
\end{equation*}

Analogously, we can prove that $\chi_{\Omega_n} v_n$ weakly converges to $u$ in $L^p(\Omega)$. Therefore, we obtain that
\begin{equation*}
\underset{n\rightarrow \infty}{\underline\lim}\int_{\Omega^n_\varepsilon}|v_n|^p\,\de\La_2\geq\underset{n\rightarrow \infty}{\underline\lim}\int_{\Omega_n}|v_n|^p\,\de\La_2 \geq \int_\Omega |u|^p\,\de\La_2.
\end{equation*}

In the notations of the limsup part of the proof, we now introduce an operator $\mathcal{M}_\varepsilon\colon C^1(\overline\Sigma_\varepsilon)\to C^{0,1}(T^0)$ defined for every $0<\varepsilon<\varepsilon_0$. Let $L_\varepsilon$ be the segment given by the intersection between $\Sigma_{l,\varepsilon}$ and the straight line orthogonal to $S_l$ at the point $(x_l,y_l)$ and let $\ell_\varepsilon$ be its length. Hence, if $P=(x,y)\in T^0\setminus V_0$, we set 
\begin{equation}\label{operatore di media}
\mathcal{M}_\varepsilon(h(x_l,y_l))=\frac{1}{\ell_\varepsilon}\int_{L_\varepsilon} h(x(\ell),y(\ell))\,\de\ell.
\end{equation}
If $P\in V_0$, we put $\mathcal{M}_\varepsilon(h(P))=h(P)$. With this definition, it holds that $\mathcal{M}_\varepsilon(h)\in C^{0,1}(T^0)$.\\
We now set
\begin{equation}\label{definizione v tilde}
\tilde{v}_n(x,y)=\mathcal{M}_\varepsilon(v_n\circ\psi_{i|n})\circ\psi^{-1}_{i|n}(x,y).
\end{equation}
We point out that $v_n$ and $\tilde{v}_n$ coincide on the set of vertices $\mathcal{V}^n$. 
We start by proving the following
\begin{equation}\label{primastimaliminf}
\E_p^\enne[\tilde{v}_n]\leq\frac{\delta_n^{1-p}}{p}\int_{\Sigma_\varepsilon^n}|\nabla v_n|^p w_\varepsilon^n\,\de\La_2. 
\end{equation}
By proceeding as in \cite{CPAA}, we can prove that
\begin{equation}\label{intermedia}
\E_p^{(n)}[\tilde{v}_n]=\frac{4^{(p-1)n}}{p}\sum_{j=1}^{3\mathcal{N}}(\tilde{v}_n(P_{j+1})-\tilde{v}_n(P_j))^p\leq\frac{1}{p}\left(\frac{4}{3}\right)^{(p-1)n}\sum_{j=1}^{3\mathcal{N}}\int_{M_j}|\nabla\tilde{v}_n|^p\,\de\ell,
\end{equation}
where $\mathcal{N}=4^n$ and $M_j$, for $j=1,\dots,3\mathcal{N}$, are the segments of $K_n$.\\
We make the explicit computations only on the term $\psi_{i|n}(AB)$. We set $h(x,y)=(v_n\circ\psi_{i|n})(x,y)$. By a change of variables, we have that
\begin{equation*}
\int_{\psi_{i|n}(AB)}|\nabla\tilde{v}_n|^p\,\de\ell=3^{n(p-1)}\int_0^1|\nabla_x\mathcal{M}_\varepsilon(h(x,0))|^p\,\de x.
\end{equation*}
By using the definition of $\mathcal{M}_\varepsilon$ we can split the above integral as follows:
\begin{equation*}
\begin{split}
&\int_0^1|\nabla_x\mathcal{M}_\varepsilon(h(x,0))|^p\,\de x=\int_0^\frac{\varepsilon}{C_1}\left(\frac{2}{C_1x}\int_{-\frac{C_1x}{2}}^0 h(x,y)\,\de y\right)^p_x\,\de x\\[3mm]
+&\int_\frac{\varepsilon}{C_1}^{1-\frac{\varepsilon}{C_1}}\left(\frac{2}{\varepsilon}\int_{-\frac{\varepsilon}{2}}^0 h(x,y)\,\de y\right)^p_x\,\de x+\int_{1-\frac{\varepsilon}{C_1}}^1\left(\frac{2}{C_1-C_1x}\int_\frac{C_1x-C_1}{2}^0 h(x,y)\,\de y\right)^p_x\,\de x=:I_1+I_2+I_3.
\end{split}
\end{equation*}
We start by evaluating $I_2$. Let $R_1$ be the rectangle of vertices $P_1$, $P_2$, $P_3$ and $P_4$ (in the notations of the limsup part). By applying H\"older inequality, we get
\begin{equation*}
I_2\leq\frac{2}{\varepsilon}\int_\frac{\varepsilon}{C_1}^{1-\frac{\varepsilon}{C_1}}\int_{-\frac{\varepsilon}{2}}^0 |h_x(x,y)|^p\,\de y\de x.
\end{equation*}
Moreover, since on $\psi_{i|n}(R_1)$ it holds that $w_\varepsilon^n(x,y)=\frac{2\cdot3^n}{\varepsilon}$, we have
\begin{equation}\label{stimaI2}
3^{n(p-1)}I_2\leq\frac{2\cdot 3^{n(p-1)}}{\varepsilon}\int_\frac{\varepsilon}{C_1}^{1-\frac{\varepsilon}{C_1}}\int_{-\frac{\varepsilon}{2}}^0 |\nabla h(x,y)|^p\,\de y\de x\leq \int_{\psi_{i|n}(R_1)}|\nabla v_n|^p w_\varepsilon^n\,\de\La_2.
\end{equation}

We now switch to $I_1$. In the notations of the limsup part, let $T_{1,1}$ be the triangle of vertices $A$, $P_1$ and $P_3$. Let
\begin{equation*}
F(x)=\frac{2}{C_1x}\int_{-\frac{C_1x}{2}}^0 h(x,y)\,\de y.
\end{equation*}
Then we have that
\begin{equation*}
\begin{split}
F'(x)&=-\frac{2}{C_1x^2}\int_{-\frac{C_1x}{2}}^0 h(x,y)\,\de y+\frac{2}{C_1x}\left(\int_{-\frac{C_1x}{2}}^0 h_x(x,y)\,\de y+\frac{C_1}{2}h\left(x,-\frac{C_1x}{2}\right)\right)\\[2mm]
&=-\frac{2}{C_1x^2}\int_{-\frac{C_1x}{2}}^0 h(x,y)\,\de y+\frac{2}{C_1x}\int_{-\frac{C_1x}{2}}^0 h_x(x,y)\,\de y+\frac{1}{x}h\left(x,-\frac{C_1x}{2}\right).
\end{split}
\end{equation*}

Since $h(x,y)=(v_n\circ\psi_{i|n})(x,y)$ belongs to $C^1(\overline{\Omega^n_\varepsilon})$ by hypothesis, from the mean value theorem we get that
\begin{equation*}
F'(x)=\frac{2}{C_1x}\int_{-\frac{C_1x}{2}}^0 h_x(x,y)\,\de y-\frac{1}{x}\left(h(x,\xi)-h\left(x,-\frac{C_1x}{2}\right)\right),
\end{equation*}
for some $\xi\in[-\frac{C_1x}{2},0]$. Hence
\begin{equation*}
F'(x)=\frac{2}{C_1x}\int_{-\frac{C_1x}{2}}^0 h_x(x,y)\,\de y-\frac{1}{x}\int_{-\frac{C_1x}{2}}^\xi h_y(x,y)\,\de y.
\end{equation*}
Going back to $I_1$, by using the convexity of the map $t\to |t|^p$ and H\"older inequality and recalling that $\xi\in[-\frac{C_1x}{2},0]$ we get
\begin{equation*}
\begin{split}
I_1&=\int_0^\frac{\varepsilon}{C_1}|F'(x)|^p\,\de x\leq 2^{p-1}\int_0^\frac{\varepsilon}{C_1}\left[\frac{2^p}{C_1^p x^p}\left(\int^0_{-\frac{C_1x}{2}}h_x(x,y)\,\de y\right)^p+\frac{1}{x^p}\left(\int^\xi_{-\frac{C_1x}{2}}h_y(x,y)\,\de y\right)^p\,\right]\,\de x\\[3mm]
&\leq 2^{p-1}\int_0^\frac{\varepsilon}{C_1}\left[\frac{2^p}{C_1^p x^p}\int^0_{-\frac{C_1x}{2}}|h_x(x,y)|^p\,\de y\left(\frac{C_1x}{2}\right)^{p-1}+\frac{1}{x^p}\int^\xi_{-\frac{C_1x}{2}}|h_y(x,y)|^p\,\de y\left(\xi+\frac{C_1x}{2}\right)^{p-1}\,\right]\,\de x\\[3mm]
&\leq 2^{p-1}\int_0^\frac{\varepsilon}{C_1}\left[\frac{2}{C_1x}\int^0_{-\frac{C_1x}{2}}|h_x(x,y)|^p\,\de y+\frac{C_1^{p-1}}{2^{p-1}x}\int^0_{-\frac{C_1x}{2}}|h_y(x,y)|^p\,\de y\,\right]\,\de x\\[3mm]
&=\int_0^\frac{\varepsilon}{C_1}\left[\frac{2^p}{C_1x}\int^0_{-\frac{C_1x}{2}}|h_x(x,y)|^p\,\de y+\frac{C_1^{p-1}}{x}\int^0_{-\frac{C_1x}{2}}|h_y(x,y)|^p\,\de y\,\right]\,\de x\\[3mm]
&\leq\int_0^\frac{\varepsilon}{C_1}\frac{2^p+C_1^p}{C_1x}\int^0_{-\frac{C_1x}{2}}\left(|h_x(x,y)|^p+|h_y(x,y)|^p\right)\,\de y\de x.
\end{split}
\end{equation*}
Since $w_\varepsilon^n(x,y)=\frac{3^n(2^p+C_1^p)}{C_1x}$ on $\psi_{i|n}(T_{1,1})$, we have that
\begin{equation}\label{stimaI1}
3^{n(p-1)}I_1\leq 3^{n(p-1)}\int_0^\frac{\varepsilon}{C_1}\frac{2^p+C_1^p}{C_1x}\int^0_{-\frac{C_1x}{2}}|\nabla h(x,y)|^p\,\de y\de x\leq \int_{\psi_{i|n}(T_{1,1})}|\nabla v_n|^p w_\varepsilon^n\,\de\La_2.
\end{equation}

We now estimate $I_3$. Again in the notations of the limsup part, let $T_{1,2}$ be the triangle of vertices $B$, $P_2$ and $P_4$. Let
\begin{equation*}
G(x):=\frac{2}{C_1-C_1x}\int_\frac{C_1x-C_1}{2}^0 h(x,y)\,\de y.
\end{equation*}
Then
\begin{equation*}
\begin{split}
G'(x)&=\frac{2C_1}{(C_1-C_1x)^2}\int_\frac{C_1x-C_1}{2}^0 h(x,y)\,\de y+\frac{2}{C_1-C_1x}\left(\int_\frac{C_1x-C_1}{2}^0 h_x(x,y)\,\de y-\frac{C_1}{2}h\left(x,\frac{C_1x-C_1}{2}\right)\right)\\[3mm]
&=\frac{2C_1}{(C_1-C_1x)^2}\int_\frac{C_1x-C_1}{2}^0 h(x,y)\,\de y+\frac{2}{C_1-C_1x}\int_\frac{C_1x-C_1}{2}^0 h_x(x,y)\,\de y-\frac{C_1}{C_1-C_1x}h\left(x,\frac{C_1x-C_1}{2}\right).
\end{split}
\end{equation*}

Again from the mean value theorem, we get that
\begin{equation*}
G'(x)=\frac{2}{C_1-C_1x}\int_\frac{C_1x-C_1}{2}^0 h_x(x,y)\,\de y+\frac{C_1}{C_1-C_1x}\left(h(x,\xi)-h\left(x,\frac{C_1x-C_1}{2}\right)\right),
\end{equation*}
for some $\xi\in[\frac{C_1x-C_1}{2},0]$. Hence
\begin{equation*}
G'(x)=\frac{2}{C_1-C_1x}\int_\frac{C_1x-C_1}{2}^0 h_x(x,y)\,\de y+\frac{C_1}{C_1-C_1x}\int_\frac{C_1x-C_1}{2}^\xi h_y(x,y)\,\de y.
\end{equation*}

As above, we can estimate $I_3$ as follows:
\begin{equation*}
\begin{split}
&I_3=\int^1_{1-\frac{\varepsilon}{C_1}}|G'(x)|^p\,\de x\\[3mm]
&\leq 2^{p-1}\int^1_{1-\frac{\varepsilon}{C_1}}\left[\frac{2^p}{(C_1- C_1x)^p}\left(\int^0_{\frac{C_1x-C_1}{2}}h_x(x,y)\,\de y\right)^p+\frac{C_1^p}{(C_1- C_1x)^p}\left(\int^\xi_{\frac{C_1x-C_1}{2}}h_y(x,y)\,\de y\right)^p\,\right]\,\de x\\[3mm]
&\!\!\!\!\!\!\!\!\!\!\!\!\!\leq 2^{p-1}\int^1_{1-\frac{\varepsilon}{C_1}}\left[\frac{2^p(C_1-C_1x)^{p-1}}{2^{p-1}(C_1- C_1x)^p}\int^0_{\frac{C_1x-C_1}{2}}|h_x(x,y)|^p\,\de y+\frac{C_1^p}{(C_1- C_1x)^p}\int^\xi_{\frac{C_1x-C_1}{2}}|h_y(x,y)|^p\,\de y\left(\xi+\frac{C_1-C_1x}{2}\right)^{p-1}\,\right]\,\de x\\[3mm]
&\leq 2^{p-1}\int^1_{1-\frac{\varepsilon}{C_1}}\left[\frac{2}{C_1-C_1x}\int^0_{\frac{C_1x-C_1}{2}}|h_x(x,y)|^p\,\de y+\frac{C_1^p}{2^{p-1}(C_1-C_1x)}\int^0_{\frac{C_1x-C_1}{2}}|h_y(x,y)|^p\,\de y\,\right]\,\de x\\[3mm]
&=\int^1_{1-\frac{\varepsilon}{C_1}}\left[\frac{2^p}{C_1-C_1x}\int^0_{\frac{C_1x-C_1}{2}}|h_x(x,y)|^p\,\de y+\frac{C_1^{p}}{C_1-C_1x}\int^0_{\frac{C_1x-C_1}{2}}|h_y(x,y)|^p\,\de y\,\right]\,\de x\\[3mm]
&\leq\int^1_{1-\frac{\varepsilon}{C_1}}\frac{2^p+C_1^p}{C_1-C_1x}\int^0_{\frac{C_1x-C_1}{2}}\left(|h_x(x,y)|^p+|h_y(x,y)|^p\right)\,\de y\de x.
\end{split}
\end{equation*}

Since $w_\varepsilon^n(x,y)=\frac{3^n(2^p+C_1^p)}{C_1-C_1x}$ on $\psi_{i|n}(T_{1,2})$, we have that
\begin{equation}\label{stimaI3}
3^{n(p-1)}I_3\leq 3^{n(p-1)}\int^1_{1-\frac{\varepsilon}{C_1}}\frac{2^p+C_1^p}{C_1-C_1x}\int^0_{\frac{C_1x-C_1}{2}}|\nabla h(x,y)|^p\,\de y\de x\leq \int_{\psi_{i|n}(T_{1,2})}|\nabla v_n|^p w_\varepsilon^n\,\de\La_2.
\end{equation}

By iterating the reasoning of \eqref{stimaI1}, \eqref{stimaI2} and \eqref{stimaI3} to the other segments, from \eqref{intermedia} we obtain \eqref{primastimaliminf}.

By proceeding as in \cite[Proposition 3.8]{CPAA}, we have that $\tilde{v}_n$ is equi-H\"older continuous on $K_n$ with exponent $\beta=\frac{d_f}{p'}$. Hence the function $\tilde{v}_n$ is defined on the discrete set $\mathcal{V}^n$, so we extend it to a continuous function $H\tilde{v}_n$ on $K$. This extension is unique and it is obtained by constructing the discrete harmonic extension $H\tilde{v}_n|_{\mathcal{V}_\star}$ of $\tilde{v}_n|_{\mathcal{V}^n}$ to the set $\mathcal{V}_\star$ and then taking the unique continuous extension of $H\tilde{v}_n|_{\mathcal{V}_\star}$ to $K$. This iterative process is known as \emph{decimation} in the physics literature (see~\cite{kozlov},~\cite{kusuoka} and also \cite{captesi}).\\
By proceeding as in the proof of Theorem 3.5 in \cite{CPAA} (liminf part), we can prove that
\begin{equation}\label{aux2}
\E_p[u]\leq\underset{n\rightarrow \infty}{\underline\lim}\E_p[H\tilde{v}_n].
\end{equation}
From the properties of harmonic extensions, in particular it follows that
\begin{equation*}
\E_p[H\tilde{v}_n]=\E_p^{(n)}[\tilde{v}_n|_{\mathcal{V}^n}].
\end{equation*}
This, along with \eqref{primastimaliminf}, implies that
\begin{equation}\label{aux3}
\E_p[u]\leq\underset{n\rightarrow \infty}{\underline\lim}\frac{\delta_n^{1-p}}{p}\int_{\Sigma_\varepsilon^n}|\nabla v_n|^p w_\varepsilon^n\,\de\La_2.
\end{equation}
The thesis then follows from the liminf properties of the sum.

We now remove the assumption $v_n\in C^{1}(\overline{\Omega^n_\varepsilon})$ for every $n\in\N$. We have that $C^{1}(\overline{\Omega^n_\varepsilon})$ is dense in $W^{1,p}(\Omega^n_\varepsilon,a^n_\varepsilon)$ and in $L^2(\Omega^n_\varepsilon)$; hence, there exists $v_n^*\in C^1(\overline{\Omega^n_\varepsilon})$ such that for every $v_n\in W^{1,p}(\Omega^n_\varepsilon,a^n_\varepsilon)$
\begin{equation}\label{densita}
\left|\Phi_\varepsilon^\enne[v_n^*]-\Phi_\varepsilon^\enne[v_n]\right|\leq\frac{1}{n}\quad\text{and}\quad \|v_n^*-v_n\|_{L^2(\Omega^n_\varepsilon)}\leq\frac{1}{n}.
\end{equation}
Let now $v_n$ be weakly converging to $u$ in $L^2(\Omega^*)$. We set
\begin{equation}\notag
\bar v^*_n:=
\begin{cases}
v_n^* \quad &\text{in }\Omega_\varepsilon^n,\\
u &\text{in }\Omega^*\setminus\Omega_\varepsilon^n.
\end{cases}
\end{equation}
We point out that \eqref{densita} implies that also $\bar v_n^*$ converges weakly to $u$ in $L^2(\Omega^*)$. Indeed, for every $\varphi\in L^2(\Omega^*)$
\begin{equation*}
\begin{split}
&\int_{\Omega^*}(\bar v_n^*-u)\varphi\,\de\La_2=\int_{\Omega^*}(\bar v_n^*-v_n)\varphi\,\de\La_2+\int_{\Omega^*}(v_n-u)\varphi\,\de\La_2=\int_{\Omega^*\setminus\Omega_\varepsilon^n}(u-v_n)\varphi\,\de\La_2\\[3mm]
&+\int_{\Omega_\varepsilon^n}(v_n^*-v_n)\varphi\,\de\La_2+\int_{\Omega^*}(v_n-u)\varphi\,\de\La_2.
\end{split}
\end{equation*}
The second term on the right-hand side of the above equality tends to zero from \eqref{densita} and the third term also vanishes from the weak convergence of $v_n$ to $u$ in $L^2(\Omega^*)$. As to the first term, we have that
\begin{equation}\notag
\int_{\Omega^*\setminus\Omega_\varepsilon^n}(u-v_n)\varphi\,\de\La_2=\int_{\Omega^*\setminus\Omega}(u-v_n)\varphi\,\de\La_2-\int_{\Omega_\varepsilon^n\setminus\Omega}(u-v_n)\varphi\,\de\La_2+\int_{\Omega\setminus\Omega_\varepsilon^n}(u-v_n)\varphi\,\de\La_2;
\end{equation}
the first term vanishes from the weak convergence, while the second and third terms vanish since $|\Omega_\varepsilon^n\setminus\Omega|\xrightarrow[n\to+\infty]{}0$ and $v_n$ is equibounded. Hence $\bar v_n^*$ converges weakly to $u$ in $L^2(\Omega^*)$.\\
Now, since $\bar v_n^*$ weakly converges to $u$ in $L^2(\Omega^*)$ and $v_n^*\in C^1(\overline{\Omega^n_\varepsilon})$, from the first part of the liminf proof we have that
\begin{equation*}
\Phi[u]\leq\underset{n\rightarrow \infty}{\underline\lim}\Phi_\varepsilon^\enne[v_n^*].
\end{equation*}
Hence, we get from \eqref{densita}
\begin{equation*}
\begin{split}
&\underset{n\rightarrow \infty}{\underline\lim}\Phi_\varepsilon^\enne[v_n]=\underset{n\rightarrow \infty}{\underline\lim}\left(\Phi_\varepsilon^\enne[v_n]-\Phi_\varepsilon^\enne[v_n^*]+\Phi_\varepsilon^\enne[v_n^*]\right)\geq\underset{n\rightarrow \infty}{\underline\lim}\left(\Phi_\varepsilon^\enne[v_n]-\Phi_\varepsilon^\enne[v_n^*]\right)+\underset{n\rightarrow \infty}{\underline\lim}\Phi_\varepsilon^\enne[v_n^*]\\[2mm]
&=0+\underset{n\rightarrow \infty}{\underline\lim}\Phi_\varepsilon^\enne[v_n^*]\geq\Phi[u].
\end{split}
\end{equation*}

\end{proof}

\section{Existence, uniqueness and convergence results}\label{exunconv}
\setcounter{equation}{0}

From now on, let $f\in L^{p'}(\Omega^*)$.

Let now 
\begin{equation*}
\Phi_\varepsilon^\enne(u,v)=\int_{\Omega^n_\varepsilon}|u|^{p-2}u\,v\,\de\La_2+\int_{\Omega^n_\varepsilon}a_\varepsilon^n(x,y)|\nabla u|^{p-2}\nabla u\nabla v\,\de\La_2
\end{equation*}
and
\begin{equation*}
\Phi(u,v)=\int_{\Omega}|u|^{p-2}u\,v\,\de\La_2+\int_{\Omega}|\nabla u|^{p-2}\nabla u\nabla v\,\de\La_2+\E_p(u,v).
\end{equation*}

We say that problem $(P_\varepsilon^n)$ formally stated before admits a weak solution $u_\varepsilon^n\in W^{1,p}(\Omega_\varepsilon^n,a^n_\varepsilon)$ if it satisfies
\begin{equation*}
\Phi_\varepsilon^\enne(u_\varepsilon^n,v)=\int_{\Omega_\varepsilon^n} fv\,\de\La_2\quad\forall\,v\in W^{1,p}(\Omega_\varepsilon^n,a^n_\varepsilon).
\end{equation*}

We recall that Proposition \ref{dscipref} in particular implies that $\Phi_\varepsilon^n$ is coercive. Hence, from Theorems 1.5.6 and 1.5.8 in \cite{badialeserra}, $\Phi_\varepsilon^n$ admits a unique minimum point $u_\varepsilon^n\in W^{1,p}(\Omega_\varepsilon^n,a^n_\varepsilon)$. Integrating by parts, we prove that such unique minimum point is also the unique weak solution of problem $(P_\varepsilon^n)$ in the following sense: for $f\in L^{p'}(\Omega^*)$ 
\begin{equation*}
\!\!\!\!\!\!\!\!\!\!\!\!\!\!\!\!\!\!\!\!\!\!\!\!\!\!\!(\overline{P_\varepsilon^n})
\begin{cases}
-\dive(a_\varepsilon^n(x,y)|\nabla u_\varepsilon^n|^{p-2}\nabla u_\varepsilon^n)+|u_\varepsilon^n|^{p-2}u_\varepsilon^n=f &\text{a.e. in}\,\,\Omega_\varepsilon^n,\\[2mm]
[u_\varepsilon^n]=0 &\text{on}\,\,K_n\text{ and on}\,\,\Gamma_\varepsilon^n,\\[2mm]
\left\langle\displaystyle|\nabla u_\varepsilon^n|^{p-2}\frac{\partial u_\varepsilon^n}{\partial\nu_1},v\right\rangle_{W^{-\frac{1}{p'},p'}(K_n),W^{\frac{1}{p'},p}(K_n)}+\left\langle\displaystyle|\nabla u_\varepsilon^n|^{p-2}\frac{\partial u_\varepsilon^n}{\partial\nu_3},v\right\rangle_{W^{-\frac{1}{p'},p'}(\partial\Lambda_\varepsilon^n),W^{\frac{1}{p'},p}(\partial\Lambda_\varepsilon^n)}\\[4mm]
+\delta_n^{1-p}\left\langle\displaystyle w_\varepsilon^n(x,y)|\nabla u_\varepsilon^n|^{p-2}\frac{\partial u_\varepsilon^n}{\partial\nu_2},v\right\rangle_{W^{-\frac{1}{p'},p'}(\partial\Sigma_\varepsilon^n),W^{\frac{1}{p'},p}(\partial\Sigma_\varepsilon^n)}=0\quad &\text{for every}\,\,v\in W^{1,p}(\Omega_\varepsilon^n),
\end{cases}
\end{equation*}

\medskip

where $\Lambda_\varepsilon^n:=\Sigma^n_{2\varepsilon}\setminus\Sigma_\varepsilon^n$ and $\nu_1$, $\nu_2$ and $\nu_3$ denote the outward unit normal vectors to $\Omega_n$, $\Sigma_\varepsilon^n$ and $\Lambda_\varepsilon^n$ respectively.



\bigskip

We say that the fractal problem admits a weak solution $u\in D(\Phi)$ if it satisfies
\begin{equation*}
\Phi(u,v)=\int_{\Omega} fv\,\de\La_2\quad\forall\,v\in D(\Phi).
\end{equation*}

We recall that Proposition \ref{dsci} in particular implies that $\Phi$ is coercive. As in the pre-homogenized case, $\Phi$ admits a unique minimum point $u\in D(\Phi)$ which is the unique weak solution of problem $(P)$ in the following sense:
\begin{equation*}
(\bar P)\begin{cases}
-\Delta_p u+|u|^{p-2}u=f &\text{a.e. in}\,\,\Omega,\\[2mm]
\E_p(u,v)+\left\langle\displaystyle\frac{\partial u}{\partial\nu}|\nabla u|^{p-2},v\right\rangle_{_{(B^{p,p}_\alpha(K))',B^{p,p}_\alpha(K)}}=0\quad &\text{for every}\,\,v\in D(\E_p).
\end{cases}
\end{equation*}

The following convergence result holds.
\begin{theorem} Let $u_\varepsilon^n\equiv u_n$ and $u$ be the unique weak solutions of problems $(\overline{P_\varepsilon^n})$ and $(\bar P)$ respectively. Then $\chi_{\Omega_\varepsilon^n}u_n\xrightarrow[n\to+\infty]{}u$ in $L^2(\Omega^*)$. 
\end{theorem}

\begin{proof} Since $\Phi_\varepsilon^\enne$ M-converges to $\Phi$ in $L^2(\Omega^*)$, from Theorem 1.10 in \cite{Att} we have that every cluster point of the sequence $\{u_n\}$ is a minimum point for $\Phi$. Since $u$ is the unique minimum of $\Phi$, it follows that $u$ is the unique cluster point of $\{u_n\}$.\\
We point out that, by standard techniques, we get that the unique solution $u_n$ of $(\overline{P_\varepsilon^n})$ is equibounded in $W^{1,p}(\Omega_\varepsilon^n,a_\varepsilon^n)$; moreover, from the properties of $a_\varepsilon^n(x,y)$, $u_n$ is also equibounded in $W^{1,p}(\Omega_\varepsilon^n)$.\\
We denote by $\bar u_n$ the trivial extension of $u_n$ to $\Omega^*$. Since $p\geq 2$, from the above results we have that
\begin{equation}\notag
\|\bar u_n\|_{H^1(\Omega^*)}\leq C,
\end{equation}
where $C$ is a positive constant independent from $n$ and $\varepsilon$. This implies that $\bar u_n$ admits a subsequence, which we still denote by $\bar u_n$, which converges to a function $u^*$ weakly in $H^1(\Omega^*)$ and strongly in $L^2(\Omega^*)$. From the uniqueness of the cluster point of the sequence $\{u_n\}$, we have that $u^*|_{\Omega}\equiv u$, hence we have that $\chi_{\Omega_\varepsilon^n}u_n\to u$ strongly in $L^2(\Omega^*)$ as $n\to+\infty$.
\end{proof}

\bigskip

\noindent {\bf Acknowledgements.} The author has been supported by the Gruppo Nazionale per l'Analisi Matematica, la Probabilit\`a e le loro
Applicazioni (GNAMPA) of the Istituto Nazionale di Alta Matematica (INdAM).\\
The author wishes to thank Professor M. R. Lancia for suggesting the problem and for the several stimulating discussions during the preparation of this paper. The author also wishes to thank the anonymous referees for their valuable comments.

\end{document}